\documentclass[a4paper]{amsart}

\usepackage[all]{xy}
\usepackage{graphicx}
\usepackage{caption}
\usepackage[hidelinks]{hyperref}

\newtheorem{thm}{Theorem}[section]
\newtheorem{pro}[thm]{Proposition}

\newtheorem{lem}[thm]{Lemma}

\theoremstyle{definition}

\newtheorem{rem}[thm]{Remark}       %%%%% the counter [thm] is optional
\newtheorem{defn}[thm]{Definition}  %%%%% here and below, but preferred
\newtheorem{exam}[thm]{Example}     %%%%% by the journal

\newcommand{\C}{\mathbb C}
\newcommand{\D}{\mathbb D}
\newcommand{\R}{\mathbb R}
\newcommand{\Q}{\mathbb Q}
\newcommand{\N}{\mathbb N}

\newcommand{\fE}{\mathcal E}

\newcommand{\cisom}[1]{[#1]_{\tiny{isom}}}

\begin{document}

\title[Construction of the discrete hull]
 {Construction of the discrete hull for the combinatorics of a regular pentagonal tiling of the plane}

\author{Maria Ramirez-Solano}

\address{Department of Mathematics, University of Copenhagen, Universitetsparken 5,
2100 K\o benhavn \O ,
 Denmark.\newline\newline
\indent Current Address:
 Department of Mathematics and Computer Science, University of Southern Denmark,
Campusvej 55, DK-5230 Odense M, Denmark.
 }

\email{solano@imada.sdu.dk}



\thanks{Supported by the Danish National Research Foundation through the Centre
for Symmetry and Deformation (DNRF92), the Faculty of Science of the University of Copenhagen, and the Villum Foundation under the project ``Local and global structures of groups and their algebras".}

% -----------------------------------------------------------

\begin{abstract}
The article,\emph{ A ``regular" pentagonal tiling of the plane}, by P. L. Bowers and K. Stephenson defines a conformal pentagonal tiling.
This is a tiling of the plane with remarkable combinatorial and geometric properties.
However, it doesn't have finite local complexity in any usual sense, and therefore we cannot study it with the usual tiling theory.
The appeal of the tiling is that all the tiles are conformally regular pentagons. But
conformal maps are not allowable under finite local complexity. On the
other hand, the tiling can be described completely by its combinatorial
data, which rather automatically has finite local complexity.
In this paper we give a construction  of the discrete hull just from the combinatorial data.
The main result of this paper is that the discrete hull is a Cantor space.
\end{abstract}

% -----------------------------------------------------------
\maketitle

\section{\textbf{Introduction}}\label{s:Intro}
The pentagonal tiling shown in Figure \ref{f:K} is a conformal tiling of the plane, which has many interesting properties, such as self-similarity (Figure \ref{f:Kselfsimilar}).
 It has been studied by K. Stephenson, and  P. L. Bowers in \cite{StephensonBowers97}, \cite{StephensonMeomoirs}, \cite{CirclePackingStephenson}, using the theory of circle packings.  See also \cite{MRSnonFLCpentTiling}.
J. W. Cannon, W. J. Floyd, and W. R. Parry has studied this tiling  in \cite{FloydFiniteSubdivisionRules01} from the purely combinatorial point of view, meaning that the tiling is just seen as a CW-complex without a specified realization in the plane. We will refer to this CW-complex as the combinatorial tiling $K$.
In this paper, we study further the combinatorial tiling  $K$ by adapting what we can from the standard tiling theory (cf. \cite{Sadun08}).
The absence of translation (or the absence of the group of isometries) makes the construction of a discrete hull $(\Xi,d)$ for $K$ different and more complicated.
Yet, we can prove similar results as in the standard tiling theory:
By Proposition \ref{p:xicompact}, the hull is a compact topological space, where $d$ is an ultrametric (Proposition \ref{p:dprimeultrametric}).
In particular the hull is complete.
In Theorem \ref{t:XiCantor}, we show that $\Xi$ is a Cantor space, the main result of this article.
Thus $\Xi$ has uncountably many elements.
We construct as well a subdivision map $\omega:\Xi\to\Xi$, which is continuous, injective, but not surjective, by theorems \ref{t:omegadiscretecontinuous}, \ref{t:omegadiscreteinjective}, \ref{t:omegadiscretenotsurjective}, respectively.

This approach could be adapted to other examples, for instance to the combinatorial tilings with subdivision maps shown in Figure 1 in \cite{FloydExpansionII06} (no need to be pentagonal).

There exist several papers in the literature employing a combinatorial approach to substitutional tilings. For instance,
in \cite{HilionGeomReali13},  B\'edaride and Hilion define combinatorial substitutions, with one of the goals of realizing them in the hyperbolic plane.
In \cite{Frank08}, Frank exposes lines of research using symbolic substitutions and block substitutions. 
In \cite{Fernique10}, Fernique and Ollinger construct combinatorial tilings with strong hierarchical structure, 
while in \cite{Peyriere86},  Peyri\`ere investigates frequency of patterns.
However, none of these papers addresses the issues and questions investigated in the present work.
Indeed, the main purpose of this article is to provide a framework for constructing a groupoid $C^*$-algebra for the discrete hull, and for computing the cohomology groups of the continuous hull, \cite{RamirezSolanoPhDThesis}.
Our construction of the $C^*$-algebra depends on decoration of the tilings of the hull, introduced in Section \ref{s:decoratingK}.
Finally, we would like to make note of the fact that Stephenson and Bowers have recently started expanding this work to a more general setting, \cite{BowersI14}, \cite{BowersII14}.
\begin{figure}[htbp]
  \begin{minipage}[b]{0.48\linewidth}
  \captionsetup{width=0.8\textwidth}
    \centering
    \includegraphics[width=\linewidth]{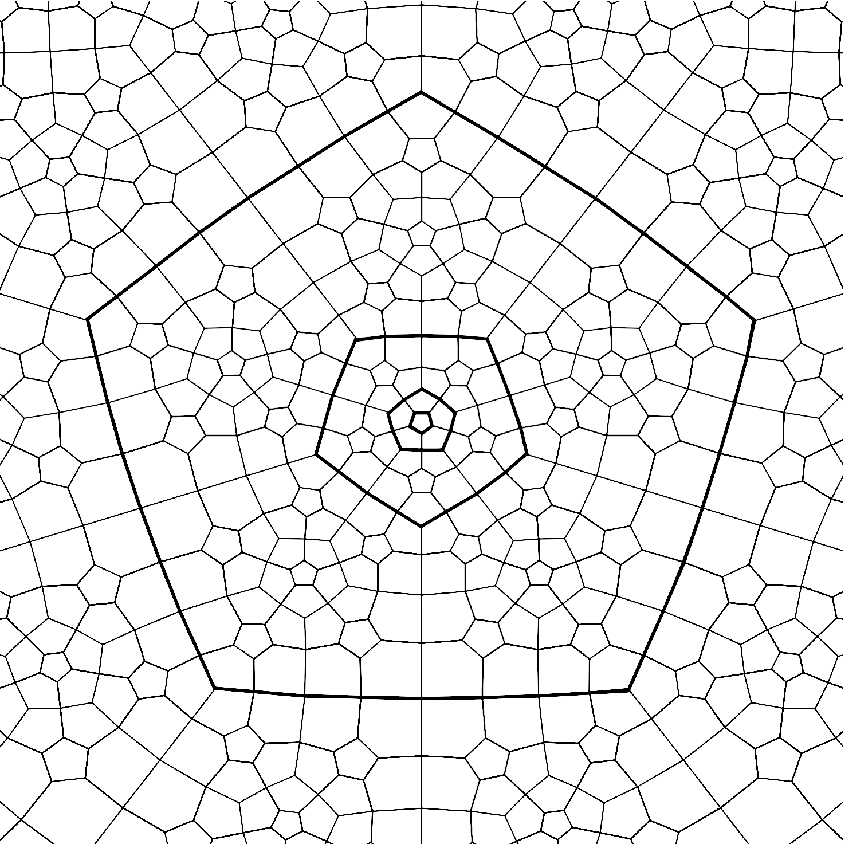}
    \caption{A conformal pentagonal tiling of the plane.}
    \label{f:K}
  \end{minipage}
  \hspace{0.2cm}
  \begin{minipage}[b]{0.48\linewidth}
  \captionsetup{width=0.8\textwidth}
    \centering
    \includegraphics[width=\linewidth]{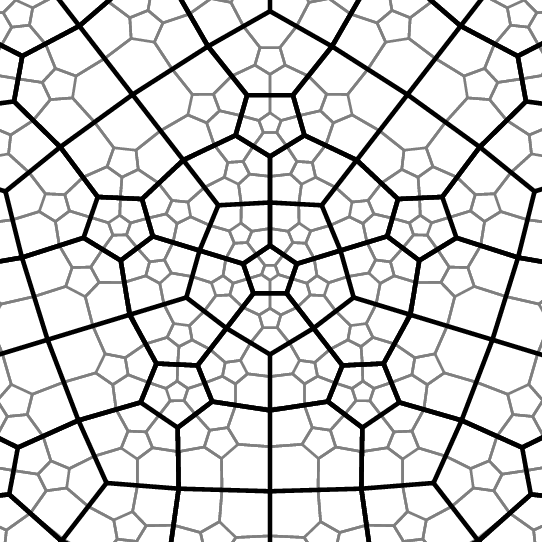}
    \caption{Selfsimilarity.\\$ $}
    \label{f:Kselfsimilar}
  \end{minipage}
\end{figure}

\section{\textbf{Combinatorial tilings}}\label{s:1CombinatorialTilings}
In this section we give the definition of combinatorial tilings coming from a subdivision rule. In particular, we give a precise definition of the combinatorial tiling $K$. Next, we show that $K$ has the so called FLC property with respect to the set of isomorphisms that are defined between subcomplexes of $K$.
We then study the so called supertiles of $K$.  After this, we redefine $K$ as a combinatorial tiling coming from a ``decorated" subdivision rule. The point of the decoration is to remove the dihedral symmetry $D_5$ of $K$. The reason for getting rid of the dihedral symmetry is so that we can construct an \'etale equivalence relation on the hull $\Xi$ and hence a $C^*$-algebra for the combinatorial tiling. See \cite{RamirezSolanoPhDThesis}, \cite{Renault80}.

A \emph{combinatorial tiling} is a 2-dimensional CW-complex  $(X,\fE)$,  such that $X$ is homeomorphic to the open unit disk $\D$, and $\fE$ a partition of $X$ satisfying the CW-complex conditions (cf. \cite{May99}).
The \emph{combinatorial tiles} (or \emph{faces}) are the closure of the 2-cells. An \emph{edge} is the closure of a 1-cell, and a \emph{vertex} is a  0-cell. We will be working with \emph{cell-preserving} maps between CW-complexes, which are  continuous maps that map cells to cells.
\begin{exam}
 If $T$ is a tiling of the plane by polygons meeting full edge to full edge, then $T$ has the structure of a 2-dimensional CW-complex, where the 2-cells are the interior of the tiles, the 1-cells are the interior of the edges of the tiles, and the 0-cells are the vertices of the edges of the tiles. Hence, under this identification, $(\C,T)$ is a combinatorial tiling.
\end{exam}

In the literature, often a patch is just a finite set of tiles. It is convenient here however that the patch is chain-connected:
\begin{defn}[patch]
  A \emph{patch} of a combinatorial tiling is a chain-connected subcomplex with finitely many cells which is the closure of its 2-cells.
\end{defn}

\begin{defn}[subdivision of a combinatorial tiling]
  Let $(X,\fE)$ and $(X,\fE')$ be two combinatorial tilings with same topological space $X$.
  We say that $(X,\fE')$ is a subdivision of $(X,\fE)$  if for each cell $e'\in \fE'$,  there is a cell $e\in\fE$ such that $e'\subset e$.
\end{defn}

\begin{figure}[h]
\captionsetup{width=0.6\textwidth}
  \centering
  \includegraphics[scale=.9]{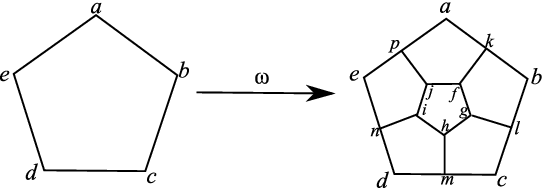}\\
  \caption{Subdivision rule for a combinatorial pentagon. }\label{f:subdivisionmapnew}
\end{figure}

\begin{defn} [pentagonal combinatorial tiling]
We say that $(X,\fE)$ is pentagonal if the closure of each 2-cell contains five 0-cells and five 1-cells.
\end{defn}

\begin{defn}[subdivision of a pentagonal tiling]\label{d:subdivisionOfAPentagonalTiling}
Given a pentagonal tiling $\fE$, we define the combinatorial tiling $\omega(\fE)$ by replacing each pentagon of $\fE$ by the rule $\omega$ shown in Figure  \ref{f:subdivisionmapnew}. More precisely,  The 0-cells of $\fE$  are 0-cells of $\omega(\fE)$.
The  1-cell $(e,a)$ from Figure \ref{f:subdivisionmapnew} subdivides into a 0-cell $p$ and two 1-cells $(e,p)$, $(p,a)$.
The 2-cell $(a,b,c,e,d)$ subdivides into five 0-cells, ten 1-cells, and six 2-cells as shown in  Figure \ref{f:subdivisionmapnew}.
\end{defn}
The \emph{subdivision of a patch of a pentagonal tiling} is defined in a similar way.

\begin{figure}[h]
  \centering
  \includegraphics[scale=.5]{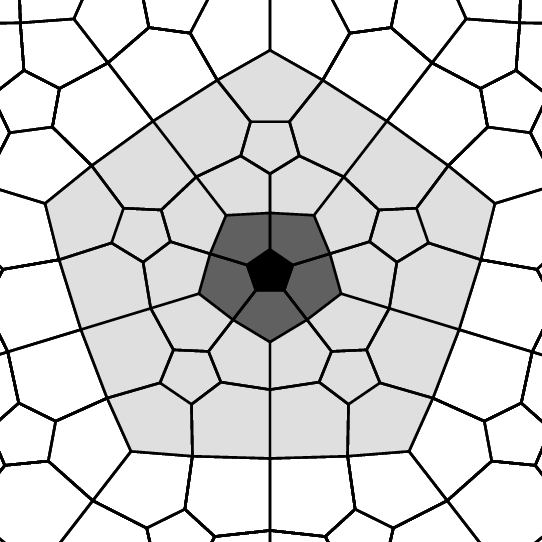}\\%alternative picture Kn2
  \caption{We call $K_0$ the black tile. $K_1$ is $K_0$ together with the dark gray tiles. $K_2$ is $K_1$ together with the light grey tiles. }\label{f:Kn}
\end{figure}

\begin{defn}[Superpentagon \emph{$K_n$}]\label{d:Kn}
Define $K_0$ as a combinatorial pentagon, which is a space homeomorphic to the closed unit disk with five distinguished points on its boundary.
Define $K_n:=\omega^n(K_0)$, $n\in\N_0$ where $\N_0:=\N\cup\{0\}$. See Figure \ref{f:Kn}.
Every $K_n$ has a distinguished central pentagon, namely the black pentagon shown in Figure \ref{f:Kn}. We define $\iota_n:K_n\to K_{n+1}$ as an embedding which maps the central pentagon of $K_n$ to the central pentagon of $K_{n+1}$.
\end{defn}
\begin{defn}[the combinatorial tiling \emph{$K$}]\label{d:combinatorialK}
  Define the complex $$K:=\lim_{n\to\infty} K_n,$$
as the direct limit of the sequence of the finite CW-complexes $K_n$ and embeddings $\iota_n$. It has a canonical CW-structure coming from the CW-structure of the complexes $K_n$, where $K_n$ is obtained from $K_{n-1}$ by attaching finitely many cells.
Each cell in the limit $K$ is the image of a cell in $K_n$ for some $n$.
\end{defn}

\subsection{\underline{\sc{Properties of $K$}}}\label{s:PropertiesoftheballK}
An Euclidean tiling of the plane is said to have finite local complexity (FLC for short) if, for any ball of radius $r$, there is a finite number of patterns of diameter less than $r$, up to elements of some fixed subgroup $G$ of the isometries of the plane, usually translations. Sometimes it is isometries. For
example, the pinwheel tiling of the plane has FLC with respect to $G$ = isometries, but not $G$ = translations.
The conformal pentagonal tiling shown in Figure \ref{f:K} does not have FLC  with respect to the set of conformal isomorphisms that are defined between open subsets of the plane \cite{MRSnonFLCpentTiling}. However, by Proposition \ref{p:Kflc}, its combinatorics $K$ has FLC with respect to the set of isomorphisms that are defined between subcomplexes of $K$.

\begin{defn}[finite local complexity (FLC)]
We say that a combinatorial tiling $L$ satisfies the \emph{finite local complexity (FLC)} if for any $r>0$, there are finitely many patches of edge-diameter less than $r$ up to the  set of isomorphisms that are defined between patches of $L$.
\end{defn}

\begin{pro}\label{p:Kflc}
  The combinatorial tiling $K$ is FLC.
\end{pro}

\begin{proof}
Given $r>0$, there is clearly a bound on the number of cells of radius $r$. Hence, there exists only a finite number of combinatorial structures. Hence $K$ is $FLC$.
\end{proof}

Any two vertices of a combinatorial tiling $L$ can be joined with finite paths of edges, as $L$ is simply-connected i.e. all its vertices are interior.
The length of a path is its number of edges.
The distance between two vertices of $L$ is defined as the length of the shortest path between them. We refer to these paths by \emph{distance-paths}.
\begin{figure}[t]
\captionsetup{width=0.6\textwidth}
  \centering
  \includegraphics[scale=.66,  angle =180]{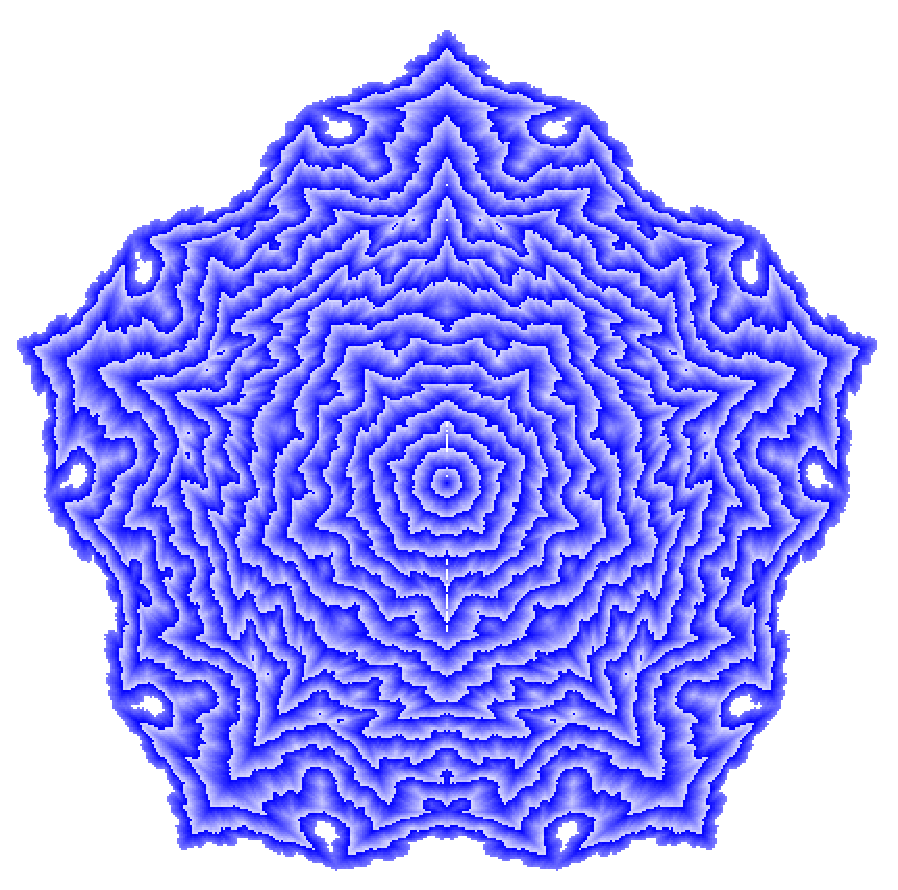}
  \caption{Concentric balls of radius $r\le 856$, where the center is the central pentagon $K_0$ of $K$. Notice the holes.}\label{f:ballholes}
\end{figure}

\begin{defn}[Ball $B(v,n,L)$]\label{d:ballonL}
  We define the ball \emph{$B(v,n,L)$} as the patch of a combinatorial tiling $L$ whose 2-cells have the property that all its vertices are within distance $n$ of the vertex $v\in L$. The closure of the 2-cells are also part of the ball.
\end{defn}
%Since $B(v,n,L)$ is a subcomplex of $L$, it is closed in $L$.

There are finitely many distinct balls $B(v,n,K)$ of radius $n\in\N$.
The \emph{boundary of  the ball $B(v,n,K)$} is defined as those edges (together with its two vertices) satisfying the condition: if $e$ is an edge of two faces $f$, $f'$, where $f$ is in the ball, and $f'$ is not in the ball, then $e$ is on the boundary of the ball.
The vertices on the boundary of the ball $B(v,n,K)$ have either distance $n$ or $n-1$ from the center.
However, all vertices of distance $n-1$ from the center are either on the boundary or inside the ball.
The vertices of distance $n$ from the center can be on the boundary, inside the ball or outside the ball (at most one unit away from the boundary).
All balls $B(v,n,K)$ are chain connected, but not necessarily simply connected. See Figure \ref{f:ballholes}. This happens simply because it is faster to go through vertices of degree 4 than vertices of degree 3.
The shortest path between two vertices goes through at least $n/2$ pentagons and at most $2n$ pentagons.

\subsection{\underline{\sc{Supertiles of $K$}}}
The vertices of $K$  have either degree 3 or degree 4.
All the faces of $K$ are of course pentagons.  But when we specify the degree on their vertices then there are exactly three choices, namely those shown in Figure \ref{f:prototilesnodec}.
We refer to these three pentagons with specified degree on their vertices by $t_1,$ $t_2$, and $t_3$ as in the figure.
\begin{figure}[h]
  \centering
  \includegraphics[scale=1.1]{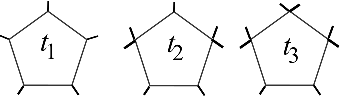}
  \caption{ For $K$, there are only three pentagons with specified vertex degree, namely those shown in the figure. We call these the prototiles of $K$.}\label{f:prototilesnodec}
\end{figure}
Notice that $t_1=K_0$. Let $t$ denote any of the three pentagons $t_1$, $t_2$, $t_3$.
We call $\omega^n(t)$, $n\in\N$ a \emph{superpentagon} of degree $n$.
We call $\omega(t)$ the \emph{flower} of $t$, and the pentagons forming the flower are called \emph{petals}.
The superpentagon $\omega^n(t)$, $n\ge 2$ can always be seen as a superflower composed of six superpentagons (which we call \emph{superpetals}) $\omega^{n-1}(t'_i)$, for some $t'_i\in\{t_1,t_2,t_3\}$ $i=1,\ldots,6$.
Given a superflower, it makes no difference whether we subdivide the superflower first and then recognize its superpetals, or if we subdivide first the superpetals individually and then form the subdivided superflower. See Figure \ref{f:omegapiomegaip}. This observation proves crucial for showing uniqueness of the decorated $K$.
This observation corresponds to  the so called ``local reflections"  in \cite{StephensonBowers97} for $t=t_1=K_0$.
A more obvious observation is that any two superpentagons of same degree are identical except on the ``corners" of each of the two superpentagons.
The degrees of the ``corners" of a superpentagon $\omega^n(t)$ are exactly the degrees of the vertices of $t$. See Figure \ref{f:isomorphicsuperpentagons}.
\begin{figure}[h]
  \captionsetup{width=0.85\textwidth}
  \centering
  \includegraphics[scale=.5]{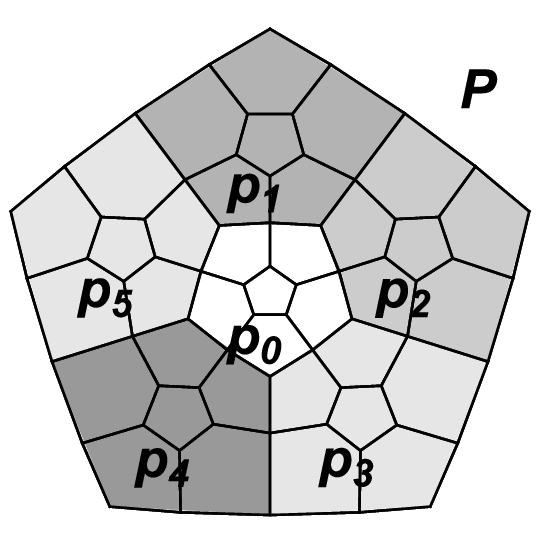} \qquad
  \includegraphics[scale=.5]{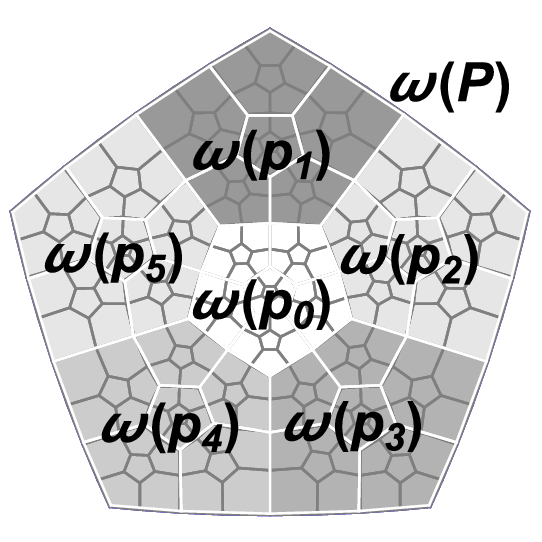}\\ %alternative omegaofsuperpentagonP.eps
  \caption{For a superpentagon $P$, the superpetals of the subdivided superpentagon $\omega(P)$ are the subdivision of the superpetals $p_i$ $i=0,\ldots,5$  of $P$.}\label{f:omegapiomegaip}
\end{figure}
\begin{figure}[ht]
\captionsetup{width=0.85\textwidth}
  \centering
  \includegraphics[scale=.5]{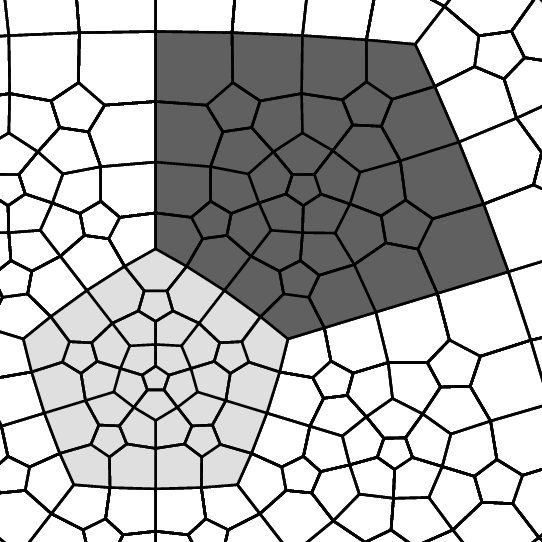}
  \caption{The light-gray superpentagon is isomorphic to the dark-gray superpentagon preserving all vertex degree except those on the ``corners" of the superpentagons. Notice that the light-gray (resp. dark-gray)  superpentagon is $\omega^2(t_1)$ (resp. $\omega^2(t_2)$) cf. Figure \ref{f:prototilesnodec}. The vertex degree of the corners of the superpentagon $\omega^2(t_1)$ (resp. $\omega^2(t_2)$) come from the vertices of $t_1$ (resp. $t_2$). }\label{f:isomorphicsuperpentagons}
\end{figure}

\newpage

%As we will see, the outside decoration of the superpentagons $\tK_n$ is completely determined by the outside decoration of $\tK_0$, which is shown in Figure \ref{tildesubdivision}.
\subsection{\underline{\sc{Decorating $K$}}}\label{s:decoratingK}

\begin{defn}[decoration of a pentagon]
The decoration of a pentagon is a bijection from its vertices to $\{1,2,3,4,5\}$ which  appear in increasing order clockwise.
\end{defn}

\begin{defn}[decorated pentagonal tiling]
A decorated pentagonal tiling is a pentagonal tiling where all its pentagons are decorated.
\end{defn}
\begin{figure}[h]
\captionsetup{width=.65\textwidth}
\centering
  \includegraphics[scale=.9]{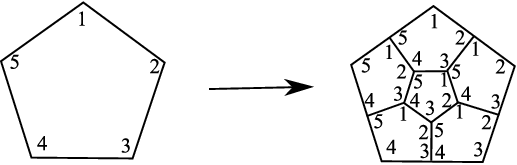}
  \caption{Subdivision rule for a decorated pentagon. We remark that the label 1 in the central pentagon is by choice.}\label{f:tildesubdivision}
\end{figure}
\begin{defn}[subdivision of a pentagonal tiling with decoration]\label{d:subdivisionOfAdecoratedPentagonalTiling}
Given a decorated pentagonal tiling $\fE$, we define the decorated tiling $\omega(\fE)$ by replacing each pentagon of $\fE$ by the subdivision rule with decoration shown in Figure \ref{f:tildesubdivision} (cf. Definition \ref{d:subdivisionOfAPentagonalTiling}).
\end{defn}

\begin{figure}[htbp]
  \begin{minipage}[b]{0.48\linewidth}
  \captionsetup{width=1.0\textwidth}
    \centering
    \includegraphics[width=\linewidth]{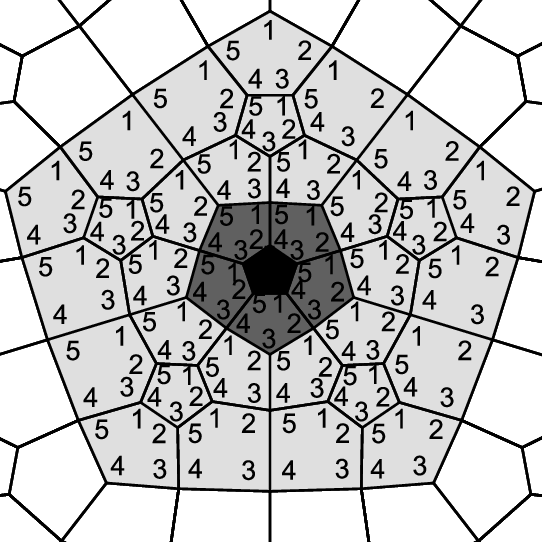}
  \caption{Decorated $K_0$, $K_1$ and $K_2$\\ (cf. Figure \ref{f:Kn}). }\label{f:Kndecorated}
  \end{minipage}
  \hspace{0.2cm}
  \begin{minipage}[b]{0.48\linewidth}
  \captionsetup{width=1.0\textwidth}
    \centering
    \includegraphics[width=\linewidth]{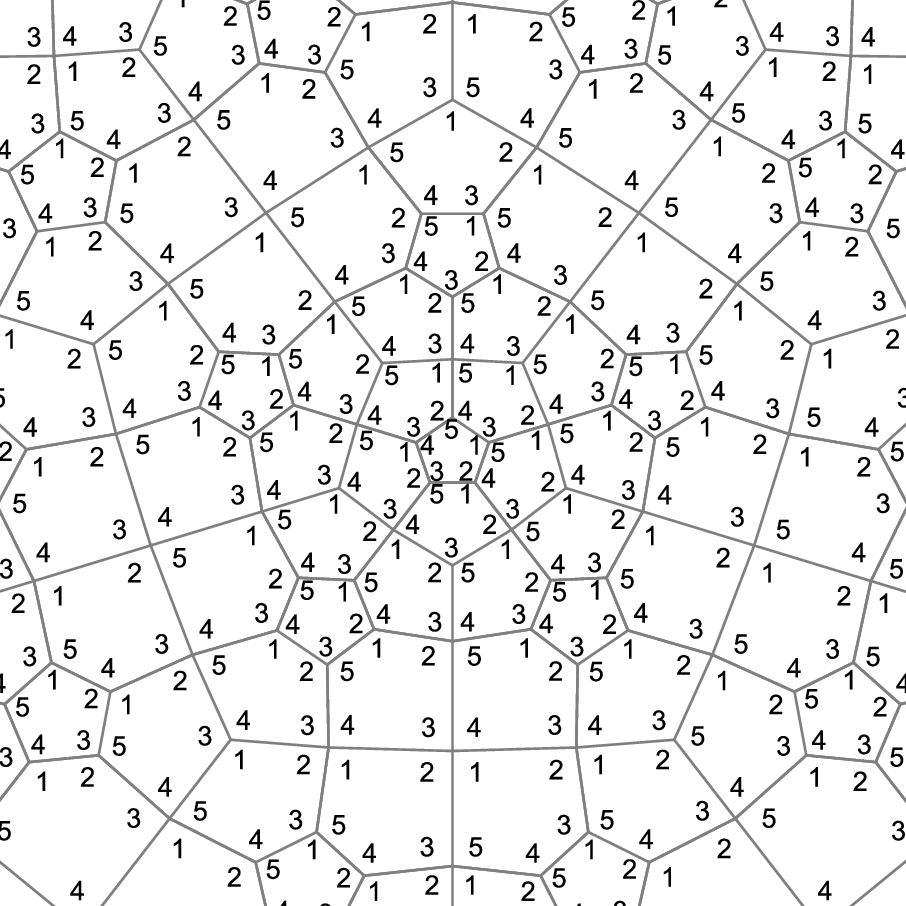}
    %\captionbox{Caption\label{fig:dummy}}{%
      %\rule{8cm}{4.5cm}
  \caption{The decorated combinatorial tiling $K$.}\label{f:decoratedK}
  \end{minipage}
\end{figure}

%\begin{figure}[h]
%\captionsetup{width=.6\textwidth}
%  \centering
%  \includegraphics[scale=.6]{./pics/10decoratedK0K1K2}\\
%  \caption{Decorated $K_0$, $K_1$ and $K_2$ (cf. Figure \ref{f:Kn}). }\label{f:Kndecorated}
%\end{figure}
\begin{defn}[Decorated superpentagon \emph{$K_n$}]\label{d:Kndecorated}
Let $K_0$ be a decorated pentagon.
Define the decorated patch $K_n:=\omega^n(K_0)$, $n\in\N_0$. See Figure \ref{f:Kndecorated}.
Notice that one vertex may get different labels from different pentagons which contain it.
Every $K_n$ has a distinguished central pentagon, namely the black pentagon shown in Figure \ref{f:Kndecorated}.
A standard induction argument shows that there is a \emph{unique} embedding  $\iota_n:K_n\to K_{n+1}$ as an embedding which maps the central pentagon of $K_n$ to the central pentagon of $K_{n+1}$ and which preserves decoration. (Note that in Definition \ref{d:Kn} we made a choice, and now we have made a unique embedding).
\end{defn}

\begin{defn}[Decorated \emph{$K$}]\label{d:combinatorialKdecorated}
 Define the complex $$K:=\lim_{n\to\infty} K_n,$$
 where each 2-cell is a decorated pentagon. See Figure \ref{f:decoratedK}. (cf. Definition \ref{d:combinatorialK}).
\end{defn}

%\begin{figure}
%\captionsetup{width=.55\textwidth}
%  \centering
%  \includegraphics[scale=.4]{./pics/11decoratedK}\\
%  \caption{The decorated combinatorial tiling $K$.}\label{f:decoratedK}
%\end{figure}

Notice that only interior edges and interior vertices are decorated in $K_n$. Eventually all edges and vertices of $K$ are decorated as all edges and vertices become interior.

\begin{thm}\label{t:Ktrivialautomorphism}
  The automorphisms of  decorated $K$ are just the identity map.
\end{thm}
\begin{proof}
Let $\phi$ be an automorphism of $K$ that preserves the decoration.  If we forget that $\phi$ preserves the decoration, then by \cite{StephensonBowers97}, $\phi$ is a rotation with respect to the central pentagon or a reflection with respect to the central pentagon and a vertex $v$. Since the decorated central pentagon has no rotations nor reflections, $\phi$ must be the identity map.
\end{proof}

\begin{figure}[h]
\captionsetup{width=0.65\textwidth}
  \centering
  \includegraphics[scale=.63]{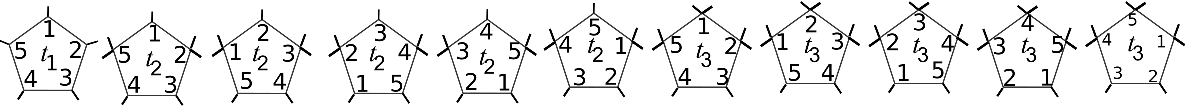}\\
  \caption{The 11 prototiles of decorated $K$. cf. Figure \ref{f:prototilesnodec}.}\label{f:prototilesdec}
\end{figure}

Decorated $K$ has eleven prototiles, i.e. eleven distinct decorated pentagons with specified degree on their vertices,  which are shown in Figure \ref{f:prototilesdec}.

\begin{figure}[h]
  \centering
  \includegraphics[scale=.8]{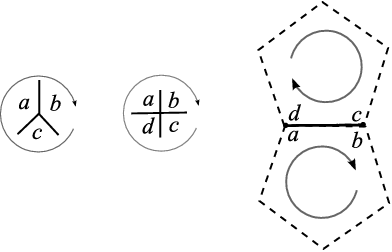}\\
  \caption{Notation of decoration of vertices and edges of decorated $K$. }\label{f:decorationnotation}
\end{figure}

We would like an analogous result of Theorem \ref{t:Ktrivialautomorphism} for certain finite subcomplexes. To do this we introduce decorations on the vertices and edges. These are induced by $K$ and are depicted  in Figure \ref{f:decorationnotation}.
A convenient notation for writing the decoration of the 3-degree vertex depicted in Figure \ref{f:decorationnotation} is $abc$ or $bca$ or $cab$ (notice the cyclic order). We will write the decoration of the 4-degree vertex from Figure \ref{f:decorationnotation} by $abcd$ or $bcda$ or $cdab$ or $dabc$.
The decoration of the edge from Figure \ref{f:decorationnotation} is for convenience written as $(ab,cd)$ or $(cd,ab)$.
 The following lemma lists all possible decorations on the edges and vertices of $K$:

\begin{lem}\label{l:possibledecorationsVE}
There are five decorations for the 3-degree vertices, for the 4-degree vertices, and for the edges of $K$. More precisely,
\begin{itemize}
  \item all the decorations of the 3-degree vertices of $K$ are  135, 124, 235, 134, 245 (notice the cyclic order),
  \item all the decorations of the 4-degree vertices of $K$ are     $1234$, $1245$, $2345$, $1235$, $1345$ (notice the cyclic order).
  \item all the decorations of the edges of $K$ are $(12,34),(12,45),(23,45)$, $(23,51)$, $(34,51)$.
\end{itemize}
\end{lem}
\begin{proof}
The decorations of edges and vertices listed in the lemma appear in $\omega^2(K_0)$.
Since no new decorations appear in $\omega^3(K_0)$, the lemma follows.
\end{proof}

The decoration of an edge tells about the  decoration of the pentagons that have in common the edge.
 The decoration of a vertex tells us as well the  decoration of the pentagons that contains them.

\begin{pro}\label{p:norotationsonpatchesofK}
 Let $v$ be a vertex of decorated $K$ and let $P$ and $Q$  be chain-connected patches  containing $v$.
  If they are isomorphic, where the isomorphism preserves decoration on all cells, and $v$ is mapped to itself then $P=Q$.
\end{pro}
\begin{proof}
 Let $\phi:P\to Q$ be an isomorphism preserving the decoration on all cells, such that $\phi(v)=v$. We call  $v$ a fixed point of $\phi$.
 If a decorated tile shares a decorated edge with a neighbor decorated tile, there is no reflection along this edge because all our decorated edges have distinct numbers on both sides. If a decorated tile shares a decorated vertex with a neighbor decorated tile,  there is no reflection along this vertex because all our decorated vertices have distinct numbers in their decoration.
 Thus since the vertex $v$ is a fixed point of the isomorphism,  the decorated faces edges and vertices having in common this vertex are also fixed by $\phi$, ie. $\phi$ is the identity map on the neighbor vertices edges and faces of $v$. Pick one of the fixed tiles of $P$ and call it $t$. Since the tile $t$ is fixed by $\phi$, and there is no reflections along edges nor vertices, $t$ and its neighbors must also be fixed  by $\phi$, i.e. $\phi$ is the identity map on the neighbor cells of $t$. By a finite induction on the neighbors, $\phi$ is the identity map.
\end{proof}

The following theorem is a corollary from the previous proposition.
\begin{thm}\label{t:autofPareid}
Let $P$ be a simply connected patch of decorated $K$. The only automorphism of $P$ preserving decoration on all cells is the identity map.
\end{thm}
\begin{proof}
Since $P$ is simply-connected, its geometric realization is the closed unit disk.
Let $\phi$ be an automorphism of $P$. By the Brouwer fixed-point theorem, $\phi$ has a fixed point $x_0$, which could either be $(i)$ a vertex, $(ii)$ be in the interior of an edge, or $(iii)$ be in an open 2-cell. In case $(i)$ the theorem  follows immediately from Proposition \ref{p:norotationsonpatchesofK}; in the case $(ii)$ the endpoints of the edge must be fixed as well by Lemma \ref{l:possibledecorationsVE}, so we are back to case $(i)$; and in case $(iii)$ the labelling of the pentagon in question forces the map to fix all the vertices of the pentagon, and we are again back to case $(i)$.
\end{proof}

\section{\textbf{The discrete Hull $\Xi$}}
The theory of $C^*$-algebras and K-theory for aperiodic  Euclidean tilings in $\R^2$ satisfying the  FLC property is well-established (cf. \cite{Sadun08}).
 An aperiodic  FLC Euclidean tiling  gives rise to a compact metric space $(\Omega,d)$ (usually called the \emph{continuous hull}) endowed with a free action of $\R^2$, and so a dynamical system  (cf. pages 5-6 in \cite{solomyak97}, \cite{Mozes89}), and its transformation groupoid $R$ (cf. Remark (ii) after Definition 1.12 of Chapter 2 in \cite{Renault80}).
 According to the Connes-Thom isomorphism, the $K$-theory of the $C^*$-algebra of this groupoid is the $K$-theory of the continuous hull $\Omega$.
 Equivalently, $\Omega$ is the classifying space of the groupoid (c.f. \cite{ConnesClassifyingSpace}) and the Baum-Connes conjecture holds since the groupoid is amenable.
 A natural transversal to this action is called the \emph{discrete hull}  (cf. page 11 in \cite{PutnamCstarKtheory00}), which we denote by $\Xi$.
 The restriction of the groupoid $R$ to $\Xi$ is an \'etale groupoid which is Morita equivalent to $R$. Hence by Theorem 2.8 in \cite{RenaultMoritaGroupoidEquivalencetoCstaralgebras} their $C^*$-algebras are strongly Morita equivalent.
 A substitution tiling is a tiling generated by a substitution rule $\omega$ with scaling factor
 $\lambda > 1$ and a finite number of prototiles, where each prototile is $\lambda$-scaled and substituted
with translation copies of the prototiles. If the substitution is primitive then the dynamical system $(\Omega,d)$ is minimal (ie. every orbit is dense), and we can construct a homeomorphism $\omega:\Omega\to\Omega$. The restriction  $\omega:\Xi\to\Xi$ is injective, continuous, but not surjective. For more details see \cite{PutnamBible95}.

 In the absence of the translation action,
 we show in this section how to construct analogues of the discrete hull for the combinatorial tiling $K$. In \cite{RamirezSolanoPhDThesis} we compute the groupoid for the discrete hull of decorated $K$ (and so a $C^*$-algebra), and analogues of the continuous hull and its topological $K$-theory (also for decorated $K$). At this point however, we have no description of the classifying space nor the groupoid for the continuous hull.

 We remark that this section applies equally to both decorated and non-decorated $K$.
The discrete hull $\Xi$ for the tiling $K$ is a topological space whose elements are basically tilings that look locally the same as $K$.  We make distinctions between elements of this space to the level of vertices, hence the use of the word discrete in the name. Equipping it with an ultrametric $d$, we show it is compact. Moreover, we define a subdivision map on it, which turns out to be  continuous, injective, but not surjective.

\begin{defn}[locally isomorphic]
 A  combinatorial tiling $L$ is \emph{locally isomorphic} to $K$ if for every patch $P$ of $L$ there is a patch $Q$ of $K$ such that $P$ and $Q$ are isomorphic, and for every patch $Q$ of $K$  there is a patch $P$ of $L$ such that $P$ and $Q$ are isomorphic.
\end{defn}
Informally, with $L$ is locally isomorphic to $K$, we mean that any finite piece of $L$ appears somewhere in  a supertile $K_n$, $n\in\N_0$, and vice versa.
Let $v$ be a  vertex of $L$, and $v'$ a vertex of $L'$. We say
$(L,v)$ is isomorphic to $(L',v')$ if there is an isomorphism $\phi:L\to L'$ with $\phi(v)=v'$.
Let $[L,v]_{isom}$ denote isomorphism classes.
The \emph{discrete hull} is defined as the set:
   $$\Xi:=\{[L,v]_{isom} \mid L \text{ is locally isomorphic to $K$, $v\in L$ a vertex}\}.$$
We will see later, (see Remark \ref{r:Recognizable}), that the tilings in the discrete hull are recognizable.
We say that $(L,u)$ is a pointed combinatorial tiling or a combinatorial tiling with origin.
(Similarly, we say that $(P,u)$ is a patch with origin $u$, and $(P,u)$ is isomorphic to $(P',u')$ if
there is an isomorphism $\phi:P\to P'$ with $\phi(u)=u'$.)

Notice that we are replacing the notion of translation $T\to T+x$ by the notion of moving the origin $(L,v)\to(L,v')$.
So periodicity in our case would become $[L,v]_{isom}=[L,v']_{isom}$.

Since any combinatorial tiling is homeomorphic to the plane, and every tiling of the plane is countable, the combinatorial tilings are countable, i.e. has countably many tiles (as each tile can be identified with a point in $\Q^2$ inside the tile).

%As a side note we remark that the counterpart notion of aperiodicity of Euclidean tilings is isomorphisms classes. In this sense, the discrete hull is composed of  aperiodic tilings.\\
%Since $(L,v)=\cup_{n\in\N}B(v,n,L)$, and every ball $B(v,n,L)$ has a  finite number of vertices, $(L,v)$ has countably many vertices.

\subsection{\underline{\sc{The metric space $(\Xi,\lowercase{d})$}}}
Recall that the ball $B(v,n,L)$ on a combinatorial tiling $L$ was introduced in Definition \ref{d:ballonL}.
For decorated $L$ we assume decoration on all cells of the ball $B(v,n,L)$.
\begin{defn}[metric $d$ on $\Xi$]
  Let $d:\Xi\times\Xi\to[0,\infty)$ be given by $$d(\cisom{L,v},\cisom{L',v'}):=\min(\frac 1n,1),$$
where $n\in\N$ is the largest radius, and the two balls $(B(v,n,L),v)\simeq (B(v',n,L'),v')$ are isomorphic.
\end{defn}

Notice that $B(v,0,L)=B(v,1,L)=\emptyset$.
Informally,  $d(\cisom{L,v},\cisom{L',v'})\le 1/n$  means that we can superimpose $(L,v)$ with $(L',v')$ at their origins $v,v'$, and they will agree on a ball of radius at least $n$.

%It is worth noting that for $n\in \N$, $B(u,n,L)\subset B(u,n+1,L)$, where $u\in L$ is any arbitrary fixed vertex,and $L=\union_{n\in\N} B(u,n,L)$.
\begin{lem}
  Let $(L,v),(L',v')$ be two combinatorial tilings locally isomorphic to $K$. If $(B(v,n,L),v)\cong (B(v',n,L'),v')$ for every integer $n\ge2$,
  then $(L,v)\cong(L',v')$.
\end{lem}
\begin{proof}
If $L$ is decorated, then the lemma is trivial, so assume $L$ is non-decorated.
For short, let $B_n:=B(v,n,L)$ and $B'_n:=B(v',n,L')$.
We have the following inclusions
$$B_2\subset B_3\subset \cdots B_n\subset\cdots\subset L,$$
$$B'_2\subset B'_3\subset \cdots B'_n\subset\cdots\subset L'$$
and the following isomorphisms $\phi_n:B_n\to B'_n$ satisfying $\phi_n(v)=v'$.
Using these maps we need to construct an isomorphism $\phi:(L,v)\to(L',v')$ such that $\phi(v)=v'$.
By definition $\phi_n(B_n)=B'_n$ and $\phi_{n+1}(B_n)\cong B'_n$ as combinatorial isomorphisms are isometric but the latter  might not be equality.
Hence we cannot use all $\phi_n$ to define $\phi$.
However, all balls $\phi_n(B_k)$, $n\in\N$ for fixed $k$ are in $L'$ and are isomorphic to $B'_k$.
Since the types of balls of radius $k$ is finite, a pattern in $\{\phi_n(B_2)\}_{n\in\N}$ must repeat infinitely many times.
Thus we can extract a subsequence $\{\phi_{\alpha_2(n)}\}_{n\in\N}$ such that all the balls $\{\phi_{\alpha_2(n)}(B_2)\}_{n\in\N}$ of radius $2$ are of the same type.
Repeating the same argument, we can extract  a subsequence $\{\phi_{\alpha_3\circ\alpha_2(n)}\}$ such that all the balls $\{\phi_{\alpha_3\circ\alpha_2}(B_2)\}_{n\in\N}$ of radius $2$ are of the same type and all the balls $\{\phi_{\alpha_3\circ\alpha_2}(B_3)\}_{n\in\N}$ of radius $3$ are of the same type.
By induction, we can extract a subsequence $\{\phi_{\alpha_k\circ\cdots\circ\alpha_2(n)}\}$ such that it gives balls of same type of radius $2,\ldots, k$.
We define $\phi$ by $\phi_{\alpha_k\circ\cdots\circ\alpha_2(n)}$,  $k\ge2$.
\end{proof}

\begin{pro}\label{p:dprimeultrametric}
  The metric $d$ on $\Xi$ is an ultrametric.
\end{pro}
\begin{proof}
1) By definition $d$ is positive.
2) We have $d(\cisom{L,v},\cisom{L,v})=0$ as $(L,v)$ agrees on itself on any ball of any radius $n$ centered at $v$.\\
3) If $d(\cisom{L,v},\cisom{L',v'})=0$ then $(B(v,n,L),v)\cong (B(v',n,L'),v')$ for any integer $n$, and therefore by the previous lemma
$\cisom{L,v}=\cisom{L',v'}.$ 4) By definition, $d(\cisom{L,v},\cisom{L',v'})=1/n=d(\cisom{L',v'},\cisom{L,v}).$
5) It remains to show the ultra triangle inequality: $d(x,z)\le \max(d(x,y),d(y,z))$, where $x,y,z\in\Xi$.
Suppose that $d(x,y)=1/n$ and $d(y,z)=1/m$. Then $x$ and $y$ agree on a ball of radius $n$, and $y$ and $z$ agree on a ball of radius $m$.
Hence $x$ and $z$ agree on a ball of radius $\min(x,y)$. Hence $d(x,z)\le 1/\min(m,n)=max(1/n,1/m)=max(d(x,y),d(y,z))$.
Since $\max(1/n,1/m)\le 1/n+1/m$, an ultrametric  is in particular a metric.
\end{proof}

%If $P$ is a patch of $L$,  $Q$ a patch of $L'$ and $P$ and $Q$ are cell-preserving isomorphic, then
\begin{lem}\label{unionofballsgivesL}
  Let $\{\cisom{L_n,v_n}\}_{n\in\N}$ be a sequence in $\Xi$.
  If $(B(v_n,n,L_n),v_n)\cong (B(v_{n+1},n,L_{n+1}),v_{n+1})$ for all integers $n\ge2$ such that $v_n$ is mapped to $v_{n+1}$,
  then there exists a  $\cisom{L,v}\in \Xi$ such that $(B(v_{n},n,L_{n}),v_{n})\cong (B(v,n,L),v)$ for all integers $n\ge2$ with $v_{n}$ mapped to $v$.
\end{lem}
\begin{proof}
Define the complex $$L=\lim_{n\to\infty} B(v_n,n,L_n),$$
as the direct limit of the sequence of balls and isomorphisms (cf. Definitions \ref{d:combinatorialK}, \ref{d:Kndecorated}, and \ref{d:combinatorialKdecorated}). It has a canonical CW-structure coming from the CW-structure of the complexes $B(v_n,n,L_n)$.
(The ball $B(v_n,n,L_n)$ is obtained from $B(v_{n-1},n-1,L_{n-1})$ by attaching finitely many cells.)
Each cell in the limit $L$ is the image of a cell in $B(v_n,n,L_n)$ for some $n$.
\end{proof}

\begin{pro}\label{p:xicompact}
  The (ultra) metric space $(\Xi,d)$ is compact.
\end{pro}
\begin{proof}
Let $\{[L_i,v_i]_{isom}\}_{i\in\N}$ be a sequence in $\Xi$. We will find a subsequence converging to some $(L,v)\in \Xi$ using a diagonal argument (cf. Lemma 1.1 in \cite{solomyak97}).

For fixed $m$, there are only finitely many distinct balls of radius $m$ by Section \ref{s:1CombinatorialTilings}.
Since $\{B(v_i,2,L_i)\}_{i\in\N}$ is an infinite number of balls of radius $2$, there is a specific type that repeats infinitely many times, say
$\{B(v_{\phi_2(i)},2,L_{\phi_2(i)})\}_{i\in\N}$, where $\phi_2:\N\to\N$ is a strictly increasing map.
Repeating the same argument on the sequence $\{(L_{\phi_2(i)},v_{\phi_2(i)})\}_{i\in\N}$, we can extract a subsequence  $$\{(L_{\phi_3(\phi_2(i))},v_{\phi_3(\phi_2(i))})\}_{i\in\N}$$ such that
all balls of radius $3$ are the same.
The map $\phi_3\circ\phi_2:\N\to\N$ is strictly increasing.
By induction we construct a subsequence
$\{(L_{\phi_n\circ\cdots\circ\phi_2(i)},v_{\phi_n\circ\cdots\circ\phi_2(i)})\}_{i\in\N}$
containing same type of balls of radius $n,n-1,\ldots,2$, where $\phi_n\circ\cdots\circ\phi_2:\N\to\N$ is a strictly increasing map.
Define $\phi(n):=\phi_n\circ\cdots\circ\phi_2(n)$, $n\ge 2$.
Then $\{(L_{\phi_n(n)},v_{\phi(n)})\}_{n\ge2}$ is a sequence containing the same type of balls of radius $m$ when $n\ge m$.
It is also a subsequence of $\{(L_i,v_i)\}_{i\in\N}$  because for $n\ge2$
\begin{eqnarray*}
  \phi(n+1)=\phi_{n+1}\circ\cdots\circ\phi_2(n+1)&>& \phi_{n+1}(\phi_n\circ\cdots\circ\phi_2(n))\\
&\ge&\phi_n\circ\cdots\circ\phi_2(n)=\phi(n).
\end{eqnarray*}
%The last inequality of type $f(x)\ge x$ is a general property of strictly increasing maps from  $\N$ to $\N$).
By Lemma \ref{unionofballsgivesL}, there is a $\cisom{L,v}\in \Xi$ such that 
$$(B(v,n,L),v)\cong(B(v_{\phi(n)},n,L_{\phi(n)}),v_{\phi(n)})$$ for all $n\ge 2$.
The subsequence $\{\cisom{L_{\phi(n)},v_{\phi(n)}}\}_{n\ge2}$ converges to $\cisom{L,v}$ because given $N\ge2$, for $n\ge N$  we have
$d(\cisom{L_{\phi(n)},v_{\phi(n)}},\cisom{L,v})\le \frac 1n\le \frac 1N$.
\end{proof}

Since $(\Xi,d)$ is a metric space, it is Hausdorff. Since $(\Xi,d)$ is also compact, it is complete and totally bounded.
By Theorem 1.58 in \cite{ultrametricAppliedAlgebraicDynamics}, every ultrametric space is totally disconnected.
Hence  $(\Xi,d)$  is a pre-Cantor space, i.e. it is compact and totally disconnected.

\newpage
\begin{defn}\label{d:tripentquadpent}
We define the tripent combinatorial tiling $T$ as
$$T:=\lim\limits_{n\to\infty} \omega^n(B(v,2,K)),$$
 with central vertex $v$ of  degree 3.
Also, we define the quadpent combinatorial tiling $Q$ as
$$Q:=\lim\limits_{n\to\infty}\omega^n(B(v',2,K)),$$
 with central vertex $v'$ of  degree 4.
See Figures \ref{f:fixedpointsofw03} and \ref{f:fixedpointsofw04}, respectively.
\end{defn}

Note that the tripent tiling has dihedral symmetry $D_3$, and the quadpent tiling has dihedral symmetry $D_4$.

\begin{figure}[t]
  \begin{minipage}[b]{0.48\linewidth}
  \captionsetup{width=1.0\textwidth}
    \centering
    \includegraphics[width=\linewidth]{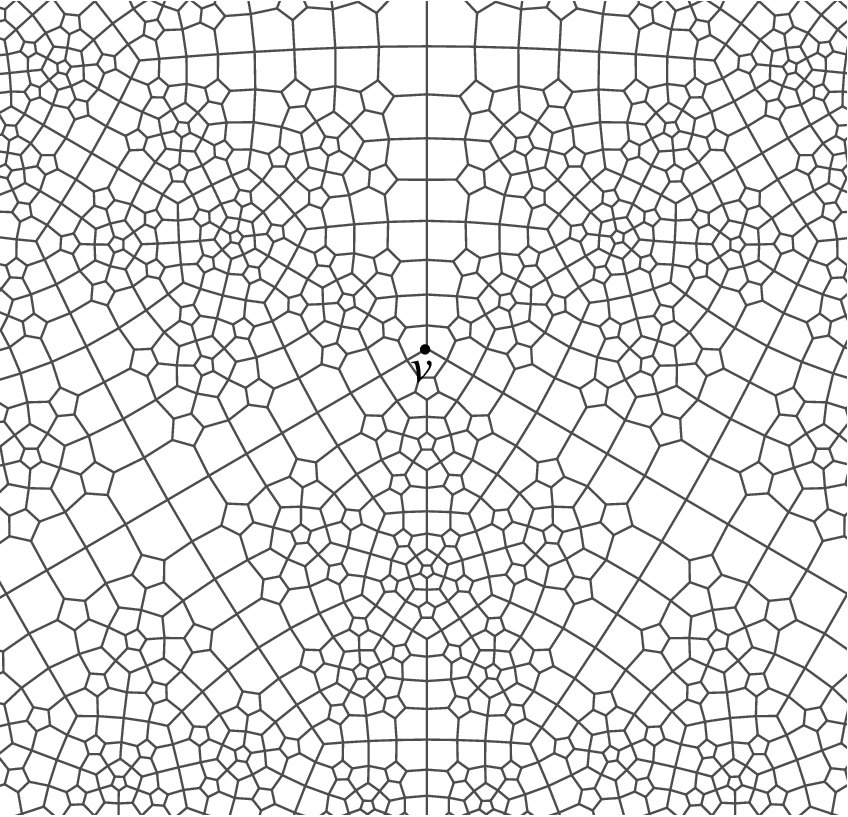}
    \caption{The tripent tiling $T$.}
    \label{f:fixedpointsofw03}
  \end{minipage}
  \hspace{0.2cm}
  \begin{minipage}[b]{0.48\linewidth}
  \captionsetup{width=1.0\textwidth}
    \centering
    \includegraphics[width=\linewidth]{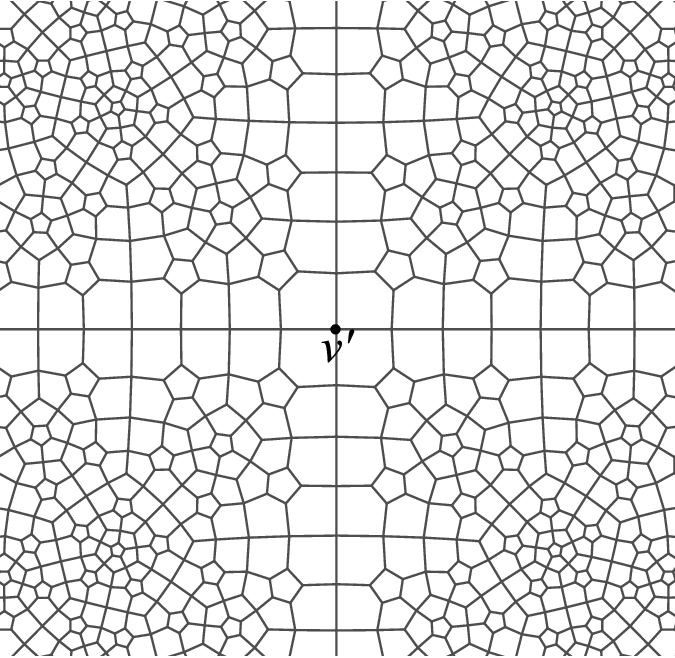}
    %\captionbox{Caption\label{fig:dummy}}{%
      %\rule{8cm}{4.5cm}
    \caption{The quadpent tiling $Q$.}
    \label{f:fixedpointsofw04}
  \end{minipage}
\end{figure}

\begin{lem}\label{l:Xitotallydisconnected}
The  space $(\Xi,d)$ has no isolated points.
\end{lem}
\begin{proof}
It suffices to show that for any $\cisom{L,v}\in \Xi$ and any $n\in\N$ there is a $\cisom{L',v'}\in\Xi$ such that $0<d(\cisom{L,v},\cisom{L',v'})\le1/n$.
 Let $n$ and $(L,v)$ be given. Since $L$ is locally isomorphic to $K$, there is a patch in $K$ isomorphic to $B(v,n,L)$.
 Let $v\in K$ also denote the image of vertex $v\in L$ under the isomorphism.
 If $(L,v)\not\cong (K,v)$ then $0<d(\cisom{L,v},\cisom{K,v})\le1/n$.
 If $(L,v)\cong (K,v)$ then instead of $K$ use the tripent tiling $T$ or the quadpent tiling $Q$.
\end{proof}

\begin{thm}\label{t:XiCantor}
  The ultrametric space $(\Xi,d)$ is a Cantor space.
\end{thm}
\begin{proof}
  The ultrametric space $(\Xi,d)$ is compact by Proposition \ref{p:xicompact}. It is totally disconnected by Theorem 1.58 in \cite{ultrametricAppliedAlgebraicDynamics} as $d$ is an ultrametric.
  It has no isolated points by the previous Lemma \ref{l:Xitotallydisconnected}.
\end{proof}

\subsection{\underline{\sc{The subdivision map $\omega$}}}
Recall that the subdivided combinatorial tiling $\omega(L)$  for a (resp. decorated) pentagonal tiling $L$ was introduced in Definition \ref{d:subdivisionOfAPentagonalTiling} (resp.  Definition \ref{d:subdivisionOfAdecoratedPentagonalTiling}).
By construction of $\omega(L)$, every vertex $v$ of $L$ is a vertex of $\omega(L)$. See Figure \ref{f:KvomegaKv}.
Define
$$\omega(L,v):=(\omega(L),v).$$
Define the subdivision map $\omega:\Xi\to\Xi$ by
$$\omega(\cisom{L,v}):=\cisom{\omega(L),v}.$$
This map is well-defined, for if $(L,v)\cong(L',v')$ so is $\omega(L,v)\cong\omega(L',v').$
\begin{figure}[htbp]
  \begin{minipage}[b]{0.48\linewidth}
  \captionsetup{width=1.0\textwidth}
    \centering
    \includegraphics[width=\linewidth]{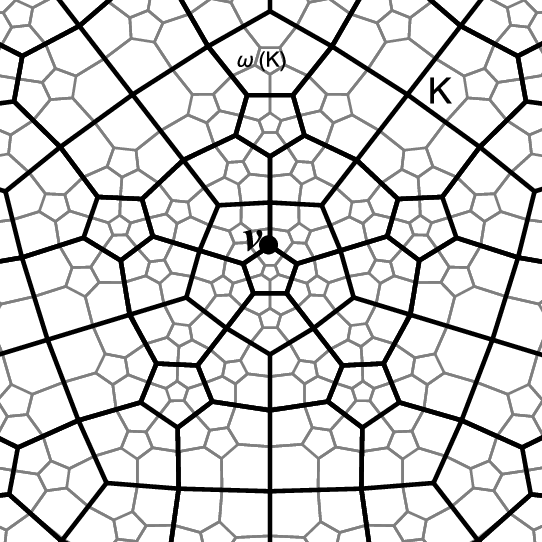}
    \caption{The combinatorial tiling $(K,v)$ is shown in thick lines and the combinatorial tiling $(\omega(K),v)$ is shown in thin lines. Both have in common the vertex $v$.}
  \label{f:KvomegaKv}
  \end{minipage}
  \hspace{0.2cm}
  \begin{minipage}[b]{0.48\linewidth}
  \captionsetup{width=1.0\textwidth}
    \centering
    \includegraphics[width=\linewidth]{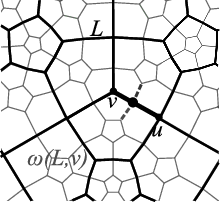}
    \caption{Reconstructing $(L,v)$ (thick) from $\omega(L,v)$ (thin). The neighbor vertex $u\in L$  of $v\in L$ is obtained from the two thick edges of $\omega(L,v)$ ignoring the incoming dotted edges of $\omega(L,v)$.}
    \label{f:w0injective}
  \end{minipage}
\end{figure}

\newpage
\begin{thm}\label{t:omegadiscreteinjective}
  The map $\omega:\Xi\to\Xi$ is injective.
\end{thm}
\begin{proof}
 Suppose that $\omega(L,v)\cong\omega(L',v')$, and let $\phi:\omega(L,v)\to\omega(L',v')$  be the isomorphism.
 We will show that $(L,v)$ is isomorphic to $(L',v')$.
 The idea of the proof is that we can recognize $(L,v)$ from $\omega(L,v)$, and $(L',v')$ from $\omega(L',v')$ in a unique way. Then $\phi$ ``restricted" to $(L,v)$ yields the isomorphism $(L,v)\cong (L',v')$.
 
 Since $v$ is a vertex of both $L$ and $\omega(L)$, and $v'$ is a vertex of both $L'$ and $\omega(L')$, $\phi$ identifies $v\in L$ with $v'\in L'$. 
 Any neighbor vertex of $v\in L$ is obtained in a unique way via $\omega(L)$ as follows:
 1) Start at $v\in \omega(L)$. 2) Go along an edge that has $v\in \omega(L)$ as a vertex. 3) Ignore the incoming edges from both sides and arrive to a new vertex $u\in \omega(L)$. This vertex is also in $u\in L$ and it is neighbor to $v\in L$. The image vertex $u':=\phi(u)\in L'$ is a neighbor vertex of $v'\in L$. (To help the reader follow this argument, see Figure \ref{f:w0injective}.)
 In this way, the map $\phi$ identifies the neighbor vertices, edges and faces of $v$ with those of $v'$.
 By a standard induction argument on the neighbor vertices, edges, and faces, $(L,v)$ is isomorphic to $(L',v')$ via $\phi$.
\end{proof}

\begin{rem}[Recognizability]\label{r:Recognizable}
The second paragraph of the proof of Theorem \ref{t:omegadiscreteinjective} shows that the tilings in the discrete hull are recognizable, i.e. that any tiling breaks into supertiles.
We would also like to point out that it has been observed earlier that injectivity is closely related to recognizability, for example in the Euclidean case see \cite{PutnamBible95}. 
\end{rem}

\begin{pro}
  For both decorated and non-decorated combinatorial tiling $K$ we have $\omega(K)\cong K$, but $\omega(K,v)\not\cong (K,v)$ for each vertex $v\in K$.
\end{pro}
\begin{proof}
By definition of  $\omega(K)$, we have $\omega(K)\cong K$.
The distance of the central pentagon of $K$ to $v$ is not the same as the distance of the central pentagon of $\omega(K)$ to $v\in\omega(K)$, for any $v\in K$.
So $\omega(K,v)\not \cong (K,v)$ for any vertex $v\in K$. This argument is illustrated in Figure \ref{f:KvomegaKv}.
\end{proof}

\begin{pro}\label{p:omegafixedpoints}
  The map $\omega:\Xi\to\Xi$ has fixed points.
\end{pro}
\begin{proof}
The tripent tiling $(T,v)$ and the quadpent tiling  $(Q,v')$, as in Definition \ref{d:tripentquadpent}, are fixed points of $\omega$. i.e. $\omega(T,v)\cong(T,v)$, $\omega(Q,v')\cong(Q,v')$.
\end{proof}

%It is worth noting that the above proof could be translated to saying that $\N$ denotes the vertices of $\omega(L)$ and $a \N$ denotes the vertices of $L$, where $a\N$ is a proper subset of $\N$.

\newpage
\begin{lem}
  If $\gamma\subset K$ is an edge-path of minimal length $n$, then $\omega(\gamma)\subset\omega(K)$ is an edge-path of minimal length $2n$.
\end{lem}
\begin{proof}
Since each edge is divided into two edges, $\omega$ doubles the length of any edge-path.
This, together with the fact that the shortest path to reach the endpoints of a subdivided edge is the subdivided edge itself, implies that $\omega$ on a path of minimal length remains a path of minimal length.
\end{proof}
\begin{figure}[htbp]
\captionsetup{width=.5\textwidth}
  \begin{minipage}[b]{\linewidth}
    \centering
    \includegraphics[width=.8\linewidth]{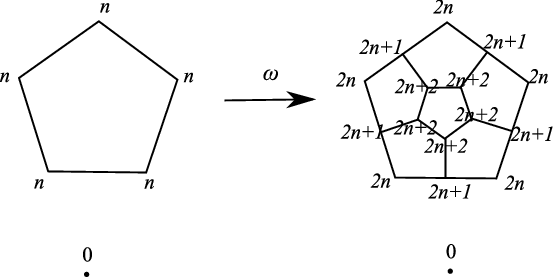}
    \caption{Lengths on a subdivided pentagon.}
    \label{f:lengthsonsubdividedpentagon}
  \end{minipage}
\end{figure}
\begin{lem}\label{l:omegaballsubsetballomega}
For any ball in $K$, we have $$B(v,{2n-2},\omega(K))\subset \omega(B(v,n,K))\subset B(v,2n+2,\omega(K)).$$
\end{lem}
\begin{proof}
  We first show that $\omega(B(v,n,K))\subset B(v,2n+2,\omega(K))$. Indeed,
  each vertex of a tile in $B(v,n,K)$ has distance at most  $n$ from the center of the ball. If all vertices of a tile are $n$-distanced, then $\omega$ on this tile will give vertices of distance at most $2n+2$. (This is illustrated in Figure \ref{f:lengthsonsubdividedpentagon}.)  Thus the vertices of each pentagon in $\omega(B(v,n,K))$ will have distance at most $2n+2$  from $v\in\omega(K)$. Hence $\omega(B(v,n,K))\subset B(v,2n+2,\omega(K))$.

  The ball $B(v,n,K)$ contains all vertices of distance $n-1$ and of smaller distance. (Recall that it contains some but not necessarily all vertices of distance $n$).
  Hence $\omega(B(v,n,K))$ contains all vertices of distance $2n-2$ (and of smaller distance) from $v\in\omega(K)$.
  Hence $\omega(B(v,n,K))$ contains the ball of radius $2n-2$ and center $v\in\omega(K)$.
\end{proof}

\begin{thm}\label{t:omegadiscretecontinuous}
The map $\omega$ is continuous.
\end{thm}
\begin{proof}
 By the previous Lemma \ref{l:omegaballsubsetballomega}, if $d(\cisom{L,v},\cisom{L',v'})=1/n$ then $$d(\omega(\cisom{L,v}),\omega(\cisom{L',v'}))\le \frac1{2n-2}.$$
% Hence for $n\ge 11$, i.e. for $d(\cisom{L,v},\cisom{L',v'})\le 1/11$,
% $$d(\omega(\cisom{L,v}),\omega(\cisom{L',v'}))\le 0.55\cdot d(\cisom{L,v},\cisom{L',v'}).$$
 Hence for $n\ge 3$, i.e. for $d(\cisom{L,v},\cisom{L',v'})\le 1/3$,
 $$d(\omega(\cisom{L,v}),\omega(\cisom{L',v'}))\le \frac34\cdot d(\cisom{L,v},\cisom{L',v'}).$$
 Thus $d$ is continuous.
\end{proof}

\newpage
\begin{thm}\label{t:omegadiscretenotsurjective}
The map $\omega:\Xi\to\Xi$ is not surjective.
\end{thm}
\begin{proof}
Let $\cisom{T,v}$ and $\cisom{Q,v'}$ be the tripent, respectively, quadpent tiling, as in Definition \ref{d:tripentquadpent}, which are fixed points of $\omega$ (cf. Proposition \ref{p:omegafixedpoints}). Define
\begin{eqnarray*}
  &&\Xi_3:=\{\cisom{L,u}\in\Xi\mid \text{ the vertex-degree of $u$ is $3$}\}\\
  &&\Xi_4:=\{\cisom{L',u'}\in\Xi\mid \text{ the vertex-degree of $u'$ is $4$}\}.
\end{eqnarray*}
If $x:=\cisom{L,u}$ is in $\Xi_3$,  then it is easy to see that $B(u,3,\omega^3(L))$ coincides with $B(v,3,T)$. Hence $d(\omega^3(x),\cisom{T,v})\le 1/3$.
Since $\cisom{T,v}$ is a fixed point of $\omega$, we get by the proof of Theorem \ref{t:omegadiscretecontinuous} that 
$$d(\omega^n(x),\cisom{T,v})\le \frac 13\left(\frac 34\right)^{n-3}\le \frac 1n, \qquad n\ge 3.$$
Hence $\omega^n(\Xi_3)\subset B_{1/n}(\cisom{T,v})$ for $n\ge 3$.
In the same way, one gets that 
$\omega^n(\Xi_4)\subset  B_{1/n}(\cisom {Q,v'})$ for $n\ge 3$.
If $\omega$ is surjective, then $w^n$ is a surjection, and since $\omega^n(\Xi_3)\subset \Xi_3$ and $\omega^n(\Xi_4)\subset \Xi_4$, it follows that 
$\omega^n(\Xi_3)=\Xi_3$ and $\omega^n(\Xi_4)=\Xi_4$.
Hence,
\begin{eqnarray*}
 &&\Xi_3\subset \bigcap_{n=3}^\infty B_{1/n}(\cisom{T,v})=\{\cisom{T,v}\} \\
 &&\Xi_4\subset \bigcap_{n=3}^\infty B_{1/n}(\cisom{Q,v'})=\{\cisom{Q,v'}\}.
\end{eqnarray*}
Therefore, $\Xi$ has only two points, which is a contradiction as $\Xi$ has uncountably many elements (it is a Cantor space).
\end{proof}

\begin{rem}[Remark to Proposition \ref{p:omegafixedpoints}]
By the proof of Theorem \ref{t:omegadiscretenotsurjective}, the only fixed points of $\omega$ are exactly the tripent tiling $(T,v)$ and the quadpent tiling $(Q,v')$ (up to decoration of $v$ and $v'$ for decorated $T$ and $Q$).
\end{rem}

\subsection*{\underline{\sc{Acknowledgments}}}
The results of this paper were obtained during my Ph.D. studies at the University of Copenhagen. I would like to express deep gratitude to Ian F. Putnam and my supervisor Erik Christensen. Special thanks go to the referee for several useful comments which helped improve readability of this work.
% -----------------------------------------------------------
\bibliographystyle{amsplain}

\end{document}